\documentclass[12pt]{amsart}
\usepackage{enumerate,amssymb,version,aliascnt,geometry}
\usepackage[all]{xy}
\usepackage{hyperref}
\usepackage{enumitem}
\usepackage{cleveref}
\hypersetup{
pdfauthor={Olivier Haution},
pdftitle={Invariants of upper motives},
pdfkeywords={upper motives, canonical dimension, Grothendieck group, algebraic cobordism},
pdfsubject={14C25 (MSC2010)},
hidelinks=true
}

\DeclareMathOperator{\Spec}{Spec}
\DeclareMathOperator{\cdim}{cdim}
\DeclareMathOperator{\im}{im}

\DeclareMathOperator{\trdeg}{tr.deg}
\DeclareMathOperator{\ind}{ind}
\DeclareMathOperator{\Td}{Td}
\DeclareMathOperator{\ch}{ch}
\DeclareMathOperator{\CH}{CH}
\DeclareMathOperator{\Br}{Br}
\DeclareMathOperator{\lgth}{length}

\newcommand{\de}[1]{d_{#1}}
\newcommand{\colim}[1]{\mycolim_{#1}}
\DeclareMathOperator*{\mycolim}{colim}
\newcommand{\dep}[1]{d_p({#1})}
\newcommand{\nep}[1]{n_p({#1})}
\newcommand{\M}[2]{{#1}_{#2}(k)}
\newcommand{\MH}[1]{H_{#1}(k)}
\newcommand{\Oc}{\mathcal{O}}
\newcommand{\Zz}{\mathbb{Z}}
\newcommand{\Qq}{\mathbb{Q}}
\newcommand{\Gc}{\mathcal{G}}
\newcommand{\Fp}{\mathbb{F}_p}
\newcommand{\Ab}{\mathsf{Ab}}

\newcommand{\G}{K_0'}
\newcommand{\K}{K_0}
\newcommand{\Gp}{K'}
\newcommand{\Kp}{K}

\newcommand{\Tan}{T}
\newcommand{\cor}{\leadsto}
\newcommand{\Reg}{\mathsf{Reg}}
\newcommand{\Vv}{\mathcal{V}}

\newtheorem{theorem}{Theorem}[section]

\newtheorem*{theorem*}{Theorem}

\newaliascnt{proposition}{theorem}
\newtheorem{proposition}[proposition]{Proposition}
\aliascntresetthe{proposition}

\newaliascnt{lemma}{theorem}
\newtheorem{lemma}[lemma]{Lemma}
\aliascntresetthe{lemma}

\newaliascnt{corollary}{theorem}
\newtheorem{corollary}[corollary]{Corollary}
\aliascntresetthe{corollary}

\theoremstyle{definition}

\newaliascnt{remark}{theorem}
\newtheorem{remark}[remark]{Remark}
\aliascntresetthe{remark}

\newaliascnt{example}{theorem}
\newtheorem{example}[example]{Example}
\aliascntresetthe{example}

\newaliascnt{definition}{theorem}
\newtheorem{definition}[definition]{Definition}
\aliascntresetthe{definition}

\newaliascnt{conditions}{theorem}
\newtheorem{conditions}[conditions]{Conditions}
\aliascntresetthe{conditions}

\begin{document}
\begin{abstract}
Let $H$ be a homology theory for algebraic varieties over a field $k$. To a complete $k$-variety $X$, one naturally attaches an ideal $\MH{X}$ of the coefficient ring $H(k)$. We show that, when $X$ is regular, this ideal depends only on the upper Chow motive of $X$. This generalises the classical results asserting that this ideal is a birational invariant of smooth varieties for particular choices of $H$, such as the Chow group. When $H$ is the Grothendieck group of coherent sheaves, we obtain a lower bound on the canonical dimension of varieties. When $H$ is the algebraic cobordism, we give a new proof of a theorem of Levine and Morel. Finally we discuss some splitting properties of geometrically unirational field extensions of small transcendence degree.
\end{abstract}
\author{Olivier Haution}
\title{Invariants of upper motives}
\email{olivier.haution at gmail.com}
\address{Mathematisches Institut, Ludwig-Maximilians-Universit\"at M\"unchen, Theresienstr.\ 39, D-80333 M\"unchen, Germany}

\thanks{This work was financed by the EPSRC Responsive Mode grant EP/G032556/1.}
\date{\today}

\maketitle

\section{Introduction}
The canonical dimension of a smooth complete algebraic variety measures to which extent it can be rationally compressed. In order to compute it, one usually studies the $p$-local version of this notion, called canonical $p$-dimension ($p$ is a prime number). In this paper, we consider the relation of $p$-equivalence \cite[\S3]{KM-standard} between complete varieties, constructed so that $p$-equivalent varieties have the same canonical $p$-dimension. This essentially corresponds to the relation of having the same upper motive with $\mathbb{F}_p$-coefficients (see \autoref{rem:upper}). For example two complete varieties $X$ and $Y$ are $p$-equivalent, for any $p$, as soon as there are rational maps $X \dasharrow Y$ and $Y \dasharrow X$. In order to obtain restrictions on the possible values of the canonical $p$-dimension of a variety, one is naturally led to study invariants of $p$-equivalence. We give a systematic way to produce such invariants (and in particular, birational invariants), starting from a homology theory. We provide examples related to $K$-theory and cycle modules. We then describe the relation between two such invariants of a complete variety $X$: its index $n_X$, and the integer $d_X$ defined as the g.c.d.\ of the Euler characteristics of the coherent sheaves of $\Oc_X$-modules. The latter invariant contains both arithmetic and geometric informations; this can be used to give bounds on the possible values of the index $n_X$ (an arithmetic invariant) in terms of the geometry of $X$. For instance, a smooth, complete, geometrically rational (or merely geometrically rationally connected, when $k$ has characteristic zero) variety of dimension $<p-1$ always has a closed point of degree prime to $p$. Some consequences of this statement are given in \autoref{cor:symbol}.

An immediate consequence of the Hirzebruch-Riemann-Roch theorem is that, for a smooth projective variety $X$ and a prime number $p$, we have
\[
\dim X\geq (p-1)( v_p(n_X) - v_p(\chi(X,\Oc_X)),
\]
(we denote by $v_p(m)$ the $p$-adic valuation of the integer $m$). When the characteristic of the base field is different from $p$, a result of Zainoulline states that $X$ is incompressible in case of equality (this follows by taking $p$-adic valuations in \cite[Corollary (A)]{Zai-09}). Here we improve this result, and obtain in \autoref{cor:cdim} the following bound on the canonical $p$-dimension of a regular complete variety $X$ over a field characteristic not $p$ (and also a weaker statement in characteristic $p$) 
\[
\cdim_p X \geq (p-1)(v_p(n_X) - v_p(d_X)).
\]
We always obtain some bound on the canonical $p$-dimension, even when it is not equal to the dimension. By contrast, the result of \cite{Zai-09}, or more generally any approach based on the degree formula (see \cite{Mer-St-03}), does not directly say anything about the canonical dimension of varieties which are not $p$-incompressible. One can sometimes circumvent this problem by exhibiting a smooth complete variety $Y$ of dimension $\cdim_p(X)$ which is $p$-equivalent to $X$, and use the degree formula to prove that $Y$ is $p$-incompressible (see \cite[\S7.3]{Mer-St-03} for the case of quadrics). But this requires to explicitly produce such a variety $Y$, and to find an appropriate characteristic number for $Y$.

Let us mention that the bound above tends to be sharp only when $\cdim_p(X)$ is not too large, compared with $p$ (see \autoref{ex:hypersurfaces}).\\

Another aspect of the technique presented here concerns its application to the algebraic cobordism $\Omega$ of Levine and Morel. Their construction uses two distinct results known in characteristic zero: \begin{enumerate}[label=({\alph*}), topsep=0pt, partopsep=0pt]\item \label{item:Hir} the resolution of singularities \cite{Hir-64}, and \item \label{item:WF} the weak factorisation theorem \cite{Weak-factorization,Weak-factorization2}.\end{enumerate} Some care is taken in the book \cite{LM-Al-07} to keep track of which result uses merely \ref{item:Hir}, or the combination of \ref{item:Hir} and \ref{item:WF}; most deeper results actually use both. In \cite{reduced}, we proved that the existence of a weak form of Steenrod operations (which may be considered as a consequence of the existence of algebraic cobordism) only uses \ref{item:Hir}. Moreover, we proved that it suffices, for the purpose of the construction of these operations, to have a $p$-local version of resolution of singularities, which has been recently obtained in characteristic different from $p$ by Gabber. By contrast, it is not clear what would be a $p$-local version of the weak factorisation theorem. In the present paper, we extend the list of results using only \ref{item:Hir} by proving the theorem below, whose original proof in \cite[Theorem~4.4.17]{LM-Al-07} was based on \ref{item:WF} (and \ref{item:Hir}). Our approach moreover allows us to give in \autoref{prop:MX} a $p$-local version of this statement, in terms of $p$-equivalence.
\begin{theorem*}
Assume that the base field $k$ admits resolution of singularities. The ideal $M(X)$ of $\Omega(k)$ generated by classes of smooth projective varieties $Y$, of dimension $<\dim X$, and admitting a morphism $Y\to X$, is a birational invariant of a smooth projective variety $X$.
\end{theorem*}

The structure of the paper is as follows. In \autoref{sect:pequivalence} we define the notion of $p$-equivalence, and describe its relation with canonical $p$-dimension. In \autoref{sect:theories} we introduce a (non-exhaustive) set of conditions on a pair of functors which are expected to be satisfied by a pair homology/cohomology of algebraic varieties. We verify these conditions for $K$-theory, cycles modules, and algebraic cobordism.  In \autoref{sect:subgroup} we explain how to construct an invariant of $p$-equivalence classes starting from any pair satisfying the conditions of \autoref{sect:theories}. We discuss each of the three situations mentioned above. In \autoref{sect:nxdx} we provide a relation between $n_X$ and $d_X$, which are invariants produced by the method of \autoref{sect:theories}. This yields the lower bound for canonical dimension. In \autoref{sect:examples}, we provide examples, and compute the lower bound in some specific situations. In \autoref{sect:ffields} we adopt the point of view of function fields, and consider splitting properties of geometrically unirational field extensions.

\section{$p$-equivalence}
\label{sect:pequivalence}
We denote by $k$ a fixed base field. A variety will be an integral, separated, finite type scheme over $\Spec k$. The function field of a variety $X$ will be denoted by $k(X)$. When $K$ is a field containing $k$, we also denote by $K$ its spectrum, and for a variety $X$ we write $X_K$ for $X \times_k K$. The letter $p$ will always denote a prime number.\\

A prime correspondence, or simply a \emph{correspondence}, $Y \cor X$ is a diagram of varieties and proper morphisms $Y \leftarrow Z \rightarrow X$, where the map $Y \to Z$ is generically finite. The degree of this map is the \emph{multiplicity} of the correspondence. 

Following \cite[\S3]{KM-standard}, we say that two varieties $X$ and $Y$ are \emph{$p$-equivalent} if there are correspondences $Y\cor X$ and $X \cor Y$ of multiplicities prime to $p$. It is equivalent to require that each of the schemes $X_{k(Y)}$ and $Y_{k(X)}$ have a closed point of degree prime to $p$.

Note that a rational map $Y \dasharrow X$ between complete varieties gives rise to a correspondence $Y \cor X$ of multiplicity $1$ (by taking for $Z$ the closure in $X \times_k Y$ of the graph of the rational map). Thus two complete varieties $X$ and $Y$ are $p$-equivalent, for all $p$, as soon as there are rational maps $Y \dasharrow X$ and $X \dasharrow Y$. Therefore we can use $p$-equivalence to find birational invariants (as in \autoref{cor:M}), but of course birational equivalence is a much finer relation (see for example \autoref{rem:birat}).

We will in general restrict our attention to complete regular varieties. In this case the relation of $p$-equivalence becomes transitive (this can be proved using \cite[Lemma~3.2]{KM-canp}). A consequence of a theorem of Gabber \cite[X, Theorem~2.1]{Gabber-book} is that, when the characteristic of $k$ is not $p$, any complete variety is $p$-equivalent to a projective regular variety (of the same dimension).\\

Let $X$ be a complete regular variety. Its canonical $p$-dimension $\cdim_p(X)$ can be defined as the least dimension of a closed subvariety $Z \subset X$ admitting a correspondence $X \cor Z$ of multiplicity prime to $p$ \cite[Corollary~4.12]{KM-canp}.

It is proven in \cite[Lemma~3.6]{KM-standard} that smooth $p$-equivalent varieties have the same canonical $p$-dimension. A slight modification of the arguments used there, together with Gabber's theorem, yields the following statement.
\begin{proposition}
Assume that the characteristic of $k$ is different from $p$. Let $X$ be a complete regular variety. Then $\cdim_p(X)$  is the least dimension of a complete regular variety $p$-equivalent to $X$.
\end{proposition}

\begin{remark}[Upper motives]
\label{rem:upper}
Let $X$ be a complete smooth variety. A summand of the Chow motive of $X$ with $\Fp$-coefficients is called \emph{upper} if it is defined by a projector of multiplicity $1 \in \Fp$ \cite[Definition~2.10]{Kar-Upper}.

If two smooth complete varieties have a common upper summand with $\Fp$-coefficients, then they are $p$-equivalent. The converse is true if the varieties are geometrically split and satisfy Rost nilpotence \cite[Corollary~2.15]{Kar-Upper}.
\end{remark}

\section{Ring theories and modules}
\label{sect:theories}
We denote by $\Ab$ the category of abelian groups. Let $\Vv$ be a full subcategory of the category of varieties and proper morphisms, and $\Reg$ a full subcategory of $\Vv$, which unless otherwise specified (in \ref{sect:chow}) will consist of the regular varieties in $\Vv$. Let $R$ and $H$ be two (covariant) functors $\Vv \to \Ab$. For a morphism $f$ of $\Vv$, we denote by $f_*$ either of the two corresponding morphisms in $\Ab$. We consider the following conditions on $R$ and $H$. 

\begin{conditions}
\label{def:main}
Let $X \in \Vv$ and $S \in \Reg$. Let $f \colon X \to S$ be a proper, dominant, generically finite morphism of degree $d$.
\begin{enumerate}[label=(R\arabic{*})]
\item \label{def:ring} The group $R(S)$ has a structure of an associative ring with unit $1_S$.

\item \label{def:u} It is possible to find an element $u\in R(X)$ such that the element $f_*u - d \cdot 1_S$ is nilpotent in the ring $R(S)$.
\end{enumerate}

\begin{enumerate}[label=(H\arabic{*})]
\item \label{def:mmodule} There is a morphism of abelian groups
\[
R(X) \otimes H(S) \to H(X) \quad ;  \quad x \otimes s \mapsto x \cdot_f s.
\]

\item \label{def:mprojformula} The group $H(S)$ has a structure of a left $R(S)$-module such that 
\[
f_*(x \cdot_f s) = (f_*x) \cdot s.
\]
\end{enumerate}
\end{conditions}

We now describe some classical examples of such functors $R$ and $H$. In all cases, $R$ and $H$ will correspond to cohomology theories, and satisfy additional properties which are logically irrelevant here. We tried to be provide minimal conditions making the proof of \autoref{lemm:main} below work.
\subsection{Quillen $K$-theory}(See \cite[\S7]{Qui-72})
\label{sect:Ktheory}Let $\Vv$ be the category of varieties and proper morphisms, and let $X \in \Vv$. We let $R(X)$ be the Grothendieck group $\G(X)$ of the category of coherent sheaves of $\Oc_X$-modules, and $H(X)$ be the $m$-th $K$-group $\Gp_m(X)$ of this category. 

We denote by $\Kp_m(X)$ the $m$-th $K$-group of the category of locally free coherent sheaves on $X$. The tensor product induces a morphism
\begin{equation}
\label{eq:modulekg}
\Kp_m(X) \otimes \Gp_0(X) \to \Gp_{m}(X) \quad ; \quad x \otimes y \mapsto x \cap y.
\end{equation}
When $S$ is regular, the map $- \cap [\Oc_S]$ induces an isomorphism
\begin{equation}
\label{eq:regular}
\varphi_S \colon \Kp_m(S) \to \Gp_m(S).
\end{equation}
With $m=0$, the combination of \eqref{eq:regular} and \eqref{eq:modulekg} gives  \ref{def:ring} (here $1_S=[\Oc_S]$). 

When $f \colon X\to S$ is a morphism, we have a morphism $f^*\colon\Kp_m(S) \to \Kp_m(X)$, and we can define
\begin{equation*}
\Gp_0(X) \otimes \Gp_m(S) \to \Gp_m(X) \quad ; \quad x \otimes s \mapsto x \cdot_f s = f^*\circ \varphi_S^{-1}(s) \cap x,
\end{equation*}
proving \ref{def:mmodule}. Then \ref{def:mprojformula} follows from the projection formula.

We prove \ref{def:u} for $u=[\Oc_X]$. The element $x = f_*u - d\cdot [\Oc_S]$ belongs to the kernel of the restriction to the generic point morphism $\G(S) \to \G(k(S))$ (see e.g.\ \cite[Lemma~2.4]{euler}). This amounts to saying that its unique antecedent $y \in \K(S)$ under \eqref{eq:regular} (with $m=0$) has rank zero. Thus for any $n$, its $n$-th power $y^n$ belongs to the $n$-th term of the gamma filtration. The image by \eqref{eq:regular} (with $m=0$) of this term is contained in the $n$-th term of the topological filtration \cite[Expos\'e~X, Corollaire~1.3.3]{sga6}. The latter vanishes when $n>\dim S$, hence so does $x^n$. 

\subsection{Chow groups and cycle modules}
\label{sect:chow}
Let us sketch how Chow groups and cycle modules can be made to fit into this framework, although the situation is somewhat degenerate. Let $\Vv$ be the category of varieties and proper morphisms, and $R$ be the Chow group $\CH$. The property \ref{def:u} is satisfied with $u=[X]$, since $f_*[X]=d \cdot [S]$.

Let $M$ be a cycle module \cite[Definition~2.1]{Rost-Chow}. We let $H(-)=A_*(-;M)$ be the Chow group with coefficients in $M$ \cite[p.356]{Rost-Chow}.

When $M$ is Quillen $K$-theory, the conditions  \labelcref{def:ring}, \labelcref{def:mmodule,def:mprojformula,} are verified in \cite[\S 8]{Gil-Ri-81} using Bloch's formula.

Taking for $\Reg$ the subcategory of smooth varieties, \ref{def:ring} is classical. When $X \to S$ is a morphism, with $S$ a smooth variety, the pairing \labelcref{def:mmodule}
\[
\CH(X) \otimes A_*(S;M) \to A_*(X;M)
\]
is defined by sending $x \otimes s$ to $g^*(x \times_k s)$. Here we use the cross product of \cite[\S 14]{Rost-Chow}, and $g^*$ is the pull-back along the regular closed embedding $g \colon X \to X \times_k S$ given by the graph of $f$, defined as the composite
\[
A_*(X \times_k S;M) \xrightarrow{J(g)} A_*(N_g;M) \xrightarrow{(p^*)^{-1}} A_*(X;M),
\]
where $J(g)$ is the deformation homomorphism \cite[\S 11]{Rost-Chow} and $p^*$ the isomorphism induced by the flat pull-back along the normal bundle $p \colon N_g=\Tan_S \times_S X \to X$ to $g$ \cite[Proposition~8.6]{Rost-Chow}. Then \labelcref{def:mprojformula} is easily verified.

\subsection{Algebraic cobordism} \label{sect:omega} Assume that $k$ admits resolution of singularities. In this paper, this will mean that $k$ satisfies the conclusion of \cite[Theorem~A.1, p.233]{LM-Al-07}. Note that it implies that $k$ is perfect. Let $\Vv$ be the category of quasi-projective varieties and projective morphisms. Then $\Reg$ is the full subcategory of smooth quasi-projective varieties. We take for $R=H$ the algebraic cobordism $\Omega$ of \cite{LM-Al-07}. We prove properties \labelcref{def:ring}, \labelcref{def:mmodule}, \labelcref{def:mprojformula}, and, under the additional assumption that $f$ is separable, we prove \labelcref{def:u}.

Property \ref{def:ring} is a classical property of $\Omega$. We now prove \ref{def:mmodule}. We have by \cite[Lemma~2.4.15]{LM-Al-07}
\[
\Omega(X) = \colim{\mathcal{C}} \Omega(Y),
\]
where the colimit is taken over the category $\mathcal{C}$ of projective $X$-schemes $Y$ which are smooth $k$-varieties; if $g \colon Y\to Z$ is a morphism in $\mathcal{C}$, the transition map is the push-forward $g_* \colon \Omega(Y) \to \Omega(Z)$. Then for $Y \in \mathcal{C}$, we can make $\Omega(S)$ act on $\Omega(Y)$ using the pull-back along $Y\to S$ and the ring structure on $\Omega(Y)$. The map $g_*$ is easily seen to be $\Omega(S)$-linear using the projection formula. This gives an action on the colimit
\begin{equation}
\label{eq:actioncolim}
\Omega(X) \otimes \Omega(S) \to \Omega(X),
\end{equation}
proving \ref{def:mmodule}. Then property \ref{def:mprojformula} follows formally from the projection formula in the smooth case.

For any $T\in \Vv$, we consider the subgroup $\Omega(T)^{(n)}$ of $\Omega(T)$ generated by the images of $g_*$, where $g$ runs over the projective morphisms $W \to T$ whose image has codimension $\geq n$ in $T$. When $T=S$ is smooth, one checks, using reduction to the diagonal and the moving lemma \cite[Proposition~3.3.1]{LM-Al-07}, that the subgroups $\{\Omega(S)^{(n)}, n\geq 0\}$ define a ring filtration on $\Omega(S)$. Since  $\Omega(S)^{(\dim X +1)}=0$, any element of $\Omega(S)^{(1)}$ is nilpotent. Let $u$ be the class in $\Omega(X)$ of any resolution of singularities of $X$. Since $f$ separable, the element $f_*u - d\cdot 1_S$ vanishes when restricted to some non-empty open subvariety of $S$ by \cite[Lemma~4.4.5]{LM-Al-07}. By the localisation sequence \cite[Theorem~3.2.7]{LM-Al-07}, this means that this element belongs to $\Omega(S)^{(1)}$, proving \ref{def:u}.\\

Let us mention that the pair $(R,H)=(\Omega(-),\Omega(-)_{(m)})$ satisfies \autoref{def:main}, with $f$ separable in \labelcref{def:u} (where $\Omega(T)_{(m)}=\Omega(T)^{(\dim T - m)}$). Indeed the main point is to see that the map \eqref{eq:actioncolim} descends to a map
\[
\Omega(X) \otimes \Omega^{(n)}(S) \to \Omega^{(n)}(X).
\]
This can be seen using the fact that pull-backs along morphisms of smooth varieties respect the filtration by codimension of supports, a consequence of the moving lemma mentioned above. When $n=1$, this can also be proved directly, using the fact that $f$ is dominant and the localisation sequence.

\section{The subgroup $\M{H}{X}$}
\label{sect:subgroup}
In this section, $(R, H)$ will be a pair of functors $\Vv \to \Ab$ satisfying \autoref{def:main}. We assume that the base $k$ belongs to $\Vv$, and therefore to $\Reg$. If $X \in \Vv$ is complete, its structural morphism $x \colon X \to k$ is then in $\Vv$. We consider the subgroup
\[
\M{H}{X}=\im\big( x_* \colon H(X) \to H(k) \big) \subset H(k).
\]
In all examples considered in this paper, $\M{H}{X}$ is actually an $R(k)$-submodule of $H(k)$.

\begin{lemma}
\label{lemm:main}
Let $S \in \Reg$ and $X \in \Vv$. Let $f \colon X \to S$ be a proper, dominant, generically finite morphism of degree $d$. Then the map 
\[
f_* \colon H(X)\otimes \Zz[1/d] \to H(S)\otimes \Zz[1/d]
\]
is surjective.
\end{lemma}
\begin{proof}
Using \labelcref{def:u}, choose $u \in R(X)$ such that $f_*u-d \cdot 1_S$ is nilpotent. The element $f_*u$ is then invertible in the ring $R(S) \otimes \Zz[1/d]$. The lemma follows, since we have by \ref{def:mprojformula}, for any $s\in H(S)$,
\[
f_* \left( u \cdot_f\left( (f_*u)^{-1} \cdot s\right) \right)=s.\qedhere
\]
\end{proof}

We denote by $\Zz_{(p)}$ the subgroup of $\Qq$ consisting of those fractions whose denominator is prime to $p$.
\begin{proposition}
\label{prop:multiplicity}
Let $X,Y$ be complete varieties, with $X \in \Vv$ and $Y \in \Reg$. Let $Y \cor X$ be a correspondence of multiplicity prime to $p$ (resp.\ let $Y\dasharrow X$ be a rational map). Then, as subgroups of $H(k)\otimes \Zz_{(p)}$ (resp.\ $H(k)$),
\[
\M{H}{Y} \otimes \Zz_{(p)} \subset \M{H}{X}\otimes \Zz_{(p)} \text{ (resp.\ } \M{H}{Y} \subset \M{H}{X}).
\]
\end{proposition}
\begin{proof}
Let $Y \leftarrow Z \to X$ be a diagram giving the correspondence (resp.\ the correspondence of multiplicity $1$ associated with the rational map). By \autoref{lemm:main}, we have
\[
\M{H}{Y} \otimes \Zz_{(p)} \subset \M{H}{Z}\otimes \Zz_{(p)} \text{ (resp.\ } \M{H}{Y} \subset \M{H}{Z}).
\]
Since there is a proper morphism $Z \to X$, we have
\[
\M{H}{Z} \subset \M{H}{X}.\qedhere
\]
\end{proof}

\begin{corollary}
\label{cor:main}
Assume that two complete varieties $Y, X \in \Reg$ are $p$-equivalent. Then, as subgroups of $H(k)\otimes \Zz_{(p)}$, we have
\[
\M{H}{Y} \otimes \Zz_{(p)} = \M{H}{X} \otimes \Zz_{(p)}.
\]
\end{corollary}

\begin{corollary}
\label{cor:birational}
Let $Y, X \in \Reg$ be two complete varieties. Assume that there are rational maps $Y \dasharrow X$ and $X \dasharrow Y$. Then $\M{H}{Y}=\M{H}{X}$ as subgroups of $H(k)$.
\end{corollary}

We now come back to the examples of theories given in \autoref{sect:theories}.

\subsection{Chow groups}
We have $\CH(k)=\Zz$, and for a complete variety $X$,
\[
\M{\CH}{X}=n_X \Zz,
\]
where $n_X$ is the index of $X$, defined as the g.c.d.\ of the degrees of closed points of $X$. We denote the $p$-adic valuation of $n_X$ by 
\[
\nep{X}=v_p(n_X).
\]
We will need the following slightly more precise version of \autoref{prop:multiplicity}, when $R=H=\CH$.
\begin{proposition}
\label{prop:multiplicitychow}
Let $X$ and $Y$ be complete varieties, with $Y$ regular. Let $Y \cor X$ be a correspondence of multiplicity $m$. Then
\[
n_X \mid m \cdot n_Y.
\]
\end{proposition}
\begin{proof}
Let $Y \xleftarrow{f} Z \xrightarrow{g} X$ be a diagram giving the correspondence. Let $y \in \CH_0(Y)$ be such that $\deg (y)=n_Y$. Using \ref{sect:chow}, we have
\[
n_X \mid \deg \circ g_*([Z] \cdot_f y) = \deg \circ f_*([Z] \cdot_f y) = \deg ((f_*[Z]) \cdot y)= \deg ( m \cdot y)=m \cdot n_Y.\qedhere
\]
\end{proof}

\subsection{Cycle modules}
Let $M$ be a cycle module, $H(-)=A_*(-;M)$, $R=\CH$ (see \S \ref{sect:chow}). We have $A_*(k;M)=M(k)$. For a complete variety $X$, we claim that the subgroup $\M{H}{X}$ of $M(k)$ is the image of the morphism
\begin{equation}
\label{eq:cyclemodule}
\bigoplus_{L/k \in \mathcal{F}_X} M(L) \to M(k),
\end{equation}
where $\mathcal{F}_X$ is the class of finite field extensions $L/k$ such that $X(L) \neq \emptyset$. Indeed the image of $A_n(X;M) \to A_n(k;M)$ vanishes when $n>0$. The group $A_0(X;M)$ is by definition a quotient of the direct sum of the groups $M(k(x))$, over all closed points $x$ of $X$. Moreover, for each such $x$, the extension $k(x)/k$ belongs to $\mathcal{F}_X$, and the map $M(k(x)) \to A_0(X;M) \to A_0(k;M)=M(k)$ is the transfer for the finite field extension $k(x)/k$. Conversely, if $L/k \in \mathcal{F}_X$, then there is a closed point $x$ of $X$ such that $k(x)/k$ is a subextension of $L/k$. Thus $M(L) \to M(k)$ factors through $M(k(x)) \to M(k)$. This proves the claim, and additionally shows that $\mathcal{F}_X$ may be replaced in \eqref{eq:cyclemodule} by the set of residue fields at closed points of $X$.

\autoref{cor:main} asserts that, when $X$ is smooth (or merely regular when $M$ is Quillen $K$-theory), the subgroup $\M{H}{X} \otimes \Zz_{(p)}$ of $M(k) \otimes \Zz_{(p)}$ only depends on the $p$-equivalence class of $X$.

\begin{remark}
\label{rem:birat}
The group $A_0(X;M)$ itself is known to be a birational invariant of a complete smooth variety $X$ (see \cite[Corollary~12.10]{Rost-Chow} and \cite[Appendix RC]{KM-standard}).
\end{remark}

\subsection{Grothendieck group}
We have $\G(k)=\Zz$, and when $X$ is a complete variety,
\[
\M{(\G)}{X}=\de{X} \Zz,
\]
for a uniquely determined positive integer $\de{X}$. The integer $\de{X}$ is the g.c.d.\ of the integers 
\[
\chi(X,\Gc)=\sum_i (-1)^i \dim_k H^i(X,\Gc),
\]
where $\Gc$ runs over the coherent sheaves of $\Oc_X$-modules. We denote the $p$-adic valuation of $\de{X}$ by 
\[
\dep{X}=v_p(\de{X}).
\]
Let us record for later reference a consequence of \autoref{prop:multiplicity}.
\begin{proposition}
\label{prop:cor}
Let $X$ and $Y$ be complete varieties, with $Y$ regular. Let $Y \cor X$ be a correspondence of multiplicity prime to $p$. Then
\[
\dep{X} \leq \dep{Y}.
\]
\end{proposition}

\subsection{Algebraic cobordism}
We say that two varieties $X$ and $Y$ are \emph{separably $p$-equivalent} if there are diagrams of projective morphisms $Y \xleftarrow{f} Z \to X$ and $X \xleftarrow{g} Z' \to Y$ such that $f$ and $g$ are both separable and generically finite of degrees prime to $p$. For a projective variety $X$, and an integer $n \geq 0$, we write $\M{\Omega}{X}^{(n)}$ for the subgroup of $\Omega(k)$ generated by classes of smooth projective varieties $Y$ admitting a morphism $Y\to X$ with image of codimension $\geq n$. In particular $\M{\Omega}{X}^{(0)}=\M{\Omega}{X}$.

\begin{proposition}
\label{prop:MX}
Assume that $k$ admits resolution of singularities. Let $Y$ and $X$ be smooth projective varieties of the same dimension, which are separably $p$-equivalent. Then, as subgroups of $\Omega(k) \otimes \Zz_{(p)}$, for $n\geq 0$,
\[
\M{\Omega}{X}^{(n)} \otimes \Zz_{(p)} = \M{\Omega}{Y}^{(n)} \otimes \Zz_{(p)} .
\]
\end{proposition}
\begin{proof}
Let $d$ be the common dimension. The proposition follows by applying \autoref{cor:main} (more precisely the analog statement for separably $p$-equivalent varieties) with $R=\Omega$ and $H=\Omega(-)_{(d-n)}$ (defined at the end of \S \ref{sect:omega}).
\end{proof}

For a projective variety $X$, let $M(X) \subset \Omega(k)$ be the ideal generated by the classes of smooth projective varieties $Y$ of dimension $<\dim X$ admitting a morphism $Y \to X$. We have $M(X) \subset \M{\Omega}{X}^{(1)}$; when $k$ admits weak factorisation, this inclusion is an equality \cite[Theorem~4.4.16]{LM-Al-07}. 

\begin{proposition}
Assume that $k$ admits resolution of singularities. Let $Y$ and $X$ be two smooth projective varieties of the same dimension. If there are rational maps $X \dasharrow Y$ and $Y \dasharrow X$, then $\M{\Omega}{X}^{(n)} = \M{\Omega}{Y}^{(n)}$ as subgroups of $\Omega(k)$, for any $n$. Moreover, $M(X)=M(Y)$ as ideals of $\Omega(k)$.
\end{proposition}
\begin{proof}
The first statement follows by taking $R=\Omega$ and $H=\Omega(-)_{(\dim X -n)}$ in \autoref{cor:birational} (which only requires the validity of \autoref{def:main} for $f$ birational, and in particular separable).

Let $V$ be a smooth projective variety of dimension $d$. The subgroup $\Omega_{<d}(k) \subset \Omega(k)$ generated by the classes of smooth projective varieties of dimension $< d$ is a direct summand of $\Omega(k)$ (indeed $\Omega(k)$ is the quotient of the free group generated by the classes of smooth projective varieties, by relations respecting the dimensional grading). Let $\pi_{<d} \colon \Omega(k) \to \Omega_{<d}(k)$ be the projection. The subgroup $\pi_{<d} \Omega_V(k)^{(1)}$ is generated by the elements $\pi_{<d}[U]$, where $U$ runs over the smooth  projective varieties admitting a non-dominant morphism to $V$. The element $\pi_{<d}[U]$ is equal to $[U]$ if $\dim U<d$, and vanishes otherwise. So $\pi_{<d} \Omega_V(k)^{(1)}$ is the subgroup of $\Omega(k)$ generated by the classes of smooth projective varieties of dimension $<d$ admitting a morphism to $V$ (necessarily non-dominant), while $M(V)$ is the ideal generated by the same elements. Thus $M(V)$ is the ideal of $\Omega(k)$ generated by $\pi_{<d} \Omega_V(k)^{(1)}$, and the second statement follows from the first, with $n=1$.
\end{proof}

The corollary below was proved in \cite[Theorem~4.4.17, 1.]{LM-Al-07} under the additional assumption that $k$ admits weak factorisation.
\begin{corollary}
\label{cor:M}
Assume that $k$ admits resolution of singularities. The ideal $M(X)$ of $\Omega(k)$ is a birational invariant of a smooth projective variety $X$.
\end{corollary}

\section{Relations between $n_X$ and $\de{X}$}
\label{sect:nxdx}
We have the obvious relation $\de{X} \mid n_X$. When $\dim X=0$, we have $n_X=\de{X}$. A consequence of the next theorem is that $\dep{X}=\nep{X}$ when $\dim X < p-1$.

\begin{theorem}
\label{th:main}
Let $X$ be a complete variety. Assume that one of the following conditions holds.
\begin{enumerate}[label=(\roman*)]
\item \label{it:char} The characteristic of $k$ is not $p$.

\item \label{it:small} We have $\dim X < p(p-1)$.

\item \label{it:regular} The variety $X$ is regular and (quasi-)projective over $k$.
\end{enumerate}

Then
\[
\nep{X} \leq \dep{X} + \Big[ \frac{\dim X}{p-1} \Big].
\]
\end{theorem}
\begin{proof}
For \ref{it:char} and \ref{it:small}, we may assume that $X$ is projective over $k$ by Chow's lemma. Indeed if $X' \to X$ is an envelope \cite[Definition~18.3]{Ful-In-98}, then it follows from \cite[Lemma~18.3]{Ful-In-98} that $n_X=n_{X'}$ and $\de{X}=\de{X'}$. The group $\G(X)$ is generated by the classes $[\Oc_Z]$, where $Z$ runs over the closed subvarieties of $X$. Therefore $d_X$ is the g.c.d.\ of the integers $\chi(Z,\Oc_Z)$, for $Z$ as above. In particular, we can find a closed subvariety $Z$ of $X$ such that 
\[
\dep{X}=v_p(\chi(Z,\Oc_Z)).
\]
We now claim that, under the assumption \ref{it:char} or \ref{it:small}, we have,
\begin{equation}
\label{eq:Z}
\nep{Z} \leq v_p(\chi(Z,\Oc_Z)) +  \Big[ \frac{\dim Z}{p-1} \Big].
\end{equation}
This will conclude the proof, since $\nep{X} \leq \nep{Z}$, and $\dim Z \leq \dim X$.

If we assume \ref{it:small}, then $\dim Z < p(p-1)$, and \eqref{eq:Z} follows from \cite[Proposition~9.1]{firstst}. But the argument of loc.\ cit.\ can also be used in the situation \ref{it:char}. Namely, we let $\tau_n \colon \G(-) \to \CH_n(-) \otimes \Qq$ be the map of \cite[Theorem~18.3]{Ful-In-98} (it is the homological Chern character, denoted $\ch_n$ in \cite{firstst}). Let $z\colon Z \to k$ be the structural morphism of $Z$. Then by \cite[Theorem~18.3 (1)]{Ful-In-98}
\begin{equation}
\label{eq:V}
\chi(Z,\Oc_Z) = \tau_0 \circ  z_*[\Oc_Z] = z_* \circ \tau_0[\Oc_Z].
\end{equation}
If \ref{it:char} holds, then \cite[Theorem~4.2]{firstst} says that the element $p^{[\dim Z/(p-1)]}\cdot \tau_0[\Oc_Z] \in \CH_0(Z) \otimes \Qq$ belongs to the image of $\CH_0(Z) \otimes \Zz_{(p)}$, hence can written as $b \otimes \lambda^{-1}$, with $\lambda$ an integer prime to $p$, and $b \in \CH_0(Z)$. Thus:
\[
\Big[ \frac{\dim Z}{p-1} \Big] + v_p(z_* \circ \tau_0[\Oc_Z]) = v_p(z_*b) \geq \nep{Z}.
\]
Using \eqref{eq:V}, this gives \eqref{eq:Z}.

Now we assume \ref{it:regular}. We identify the groups $\G(X)$ and $\K(X)$ using \eqref{eq:regular}, and take $a \in \K(X)$. Let $x \colon X \to k$ be the structural morphism of $X$. Since $x$ is quasi-projective and $X$ regular, the morphism $x$ is a local complete intersection (i.e.\ factors as a regular closed embedding followed by a smooth morphism); let $\Tan_x \in \K(X)$ be its virtual tangent bundle. We apply the Grothendieck-Riemann-Roch theorem \cite[Theorem~18.2]{Ful-In-98}, and get in $\Qq=\G(k) \otimes \Qq=\CH(k) \otimes \Qq$ the equalities
\begin{equation}
\label{eq:GRR}
x_*a=\ch \circ x_*a=x_* \circ \Td(\Tan_x) \circ \ch a.
\end{equation}
Here $\ch$ is the Chern character, with components $\ch^n \colon \K(X) \to \CH^n(X) \otimes \Qq$, and $\Td=\sum_n \Td^n$ is the Todd class. By \cite[Lemma~6.3]{firstst}, the morphism
\[
p^{[n/(p-1)]} \cdot \Td^n(\Tan_x) \colon  \CH^{\bullet}(X) \otimes \Qq \to \CH^{\bullet+n}(X) \otimes \Qq
\]
sends the image of $\CH^{\bullet}(X) \otimes \Zz_{(p)}$ to the image $\CH^{\bullet+n}(X) \otimes \Zz_{(p)}$. The degree zero component of $p^{[\dim X/(p-1)]}\cdot \Td(\Tan_X) \circ \ch a$ is
\[
\sum_{n=0}^{\dim X} \big(p^{[(\dim X-n)/(p-1)]}\cdot \Td^{\dim X-n}(\Tan_x)\big) \circ \big(p^{[n/(p-1)]} \cdot \ch^n a\big) \in \CH_0(X) \otimes \Qq.
\]
By \autoref{lemm:intch} and the remark above, this element belongs to the image of $\CH_0(X) \otimes \Zz_{(p)} \to \CH_0(X) \otimes \Qq$. Using \eqref{eq:GRR}, we obtain,
\[
\Big[ \frac{\dim X}{p-1} \Big] + v_p(x_*a) \geq \nep{X}.
\]
The statement follows, since we can choose $a$ such that $v_p(x_*a)=\dep{X}$.
\end{proof}

\begin{lemma}
\label{lemm:intch}
Let $X$ be a variety, and $\ch \colon \K(X) \to \CH(X) \otimes \Qq$ the Chern character. Then for all integers $n$ and elements $a \in \K(X)$, we have
\[
p^{[n/(p-1)]} \cdot \ch^n a \in \im\big(\!\CH^n(X) \otimes \Zz_{(p)} \to \CH^n(X) \otimes \Qq\big).
\]
\end{lemma}
\begin{proof}
This follows from the splitting principle. In more details, we proceed exactly as in \cite[Lemma~6.3]{firstst}, using that $\ch$ factors through the operational Chow ring tensored with $\Qq$, and that for a line bundle $L$,
\[
\ch [L] = \sum_{n \geq 0} \frac{c_1(L)^n}{n!}.\qedhere
\]
\end{proof}

\begin{remark}
From the proof, we see that the statement of \autoref{lemm:intch} may be improved: the exponent $[n/(p-1)]$ can be replaced by $[(n-1)/(p-1)]$.
\end{remark}

\begin{proposition}
\label{cor:cdim}
Let $X$ be a complete regular variety. Then
\[
\cdim_p(X) \geq\left\{ \begin{array}{ll}
		  (p-1)\cdot (\nep{X}-\dep{X}) &\mbox{ if $k$ has characteristic $\neq p$,} \\
		  (p-1)\cdot \min(p,\nep{X}-\dep{X}) &\mbox{ if $k$ has characteristic $=p$.}
       		\end{array}
	\right.
\]
\end{proposition}
\begin{proof}
Let $Z \subset X$ be a closed subvariety admitting a correspondence $X \cor Z$ of multiplicity prime to $p$, and such that $\dim Z=\cdim_p(X)$. By \autoref{prop:cor} we have $\dep{Z} \leq \dep{X}$. Since there is a morphism $Z \to X$, we have $\nep{X} \leq \nep{Z}$. We conclude by applying \autoref{th:main} \ref{it:char}, \ref{it:small} to the complete variety $Z$.
\end{proof}

\section{Examples}
\label{sect:examples}
In view of \autoref{sect:nxdx}, it may seem desirable to find conditions on a complete variety $X$ that give upper bounds for $\dep{X}$.

\begin{proposition}
\label{ex:rational}
Let $X$ be a complete smooth variety. Assume that there is a field extension $l/k$, and a complete smooth $l$-variety $Y$ such that $X_{l} \times_l Y$ is a rational $l$-variety. Then $\de{X}=1$.
\end{proposition}
\begin{proof}
It will be sufficient to prove that $\chi(X,\Oc_X)=\pm 1$. While doing so, we may extend scalars, and thus assume that $k=l$. The variety $Z=X \times_k Y$ is then rational. Since the coherent cohomology groups of the structure sheaf are birational invariants of a complete smooth variety \cite[Theorem~3.2.8]{Higherdirect}, so is its Euler characteristic. The structure sheaf of the projective space has Euler characteristic equal to $1$, hence $\chi(Z,\Oc_Z)=1$. Since $\chi(Z,\Oc_Z)=\chi(X,\Oc_X) \cdot \chi(Y,\Oc_Y)$ by \cite[Example~15.2.12]{Ful-In-98}, we are done.
\end{proof}
\begin{corollary}
\label{cor:rational}
Let $X$ be a complete, smooth, geometrically rational variety. Then $\de{X}=1$.
\end{corollary}

\begin{example}[Projective homogeneous varieties]
Let $X$ be a complete, smooth, geometrically connected variety, which is homogeneous under a semi-simple linear algebraic group. Then $X$ is geometrically rational. Thus by \autoref{cor:rational}, we have $\de{X}=1$.
\end{example}

\begin{proposition}
\label{ex:rationallyconnected}
Assume that $k$ has characteristic zero. Let $X$ be a complete, smooth, geometrically rationally connected variety. Then $\de{X}=1$.
\end{proposition}
\begin{proof}
We proceed as in the proof of \autoref{ex:rational}, and assume that $X$ is rationally connected. Then the groups $H^i(X,\Oc_X)$ vanish for $i>0$ by \cite[Corollary~4.18, a)]{Debarre-book}, and therefore $\chi(X,\Oc_X)=1$.
\end{proof}
\begin{remark}[Decomposition of the diagonal]
The statement of \autoref{ex:rationallyconnected} is more generally true (still in characteristic zero, for $X$ smooth complete) under the assumption that the diagonal decomposes, i.e.\ that there is a zero-cycle $z$ on $X$ whose degree $N$ is not zero, and a non-empty open subvariety $U$ of $X$, such that the cycles $[U] \times_k z$ and $N\cdot [\Gamma_U]$ are rationally equivalent on $U \times_k X$, where $\Gamma_U \subset U\times_k X$ is the graph of $U \to X$. This is the case when $\CH_0(X_\Omega)=\Zz$, where $\Omega$ is an algebraic closure of $k(X)$ (see e.g.\ the introduction of \cite{Esnault-trivialchow}).
\end{remark}

\begin{example}[Complete intersections]
\label{ex:completeinter}
Let $H$ be a complete intersection of hypersurfaces in $\mathbb{P}^n$ of degrees $\delta_1,\cdots,\delta_m$, with $\delta_1+ \cdots + \delta_m \leq n$. When $H$ is smooth, it is Fano, hence geometrically rationally chain connected. If, in addition, $k$ has characteristic zero, the variety $H$ is geometrically rationally connected, and \autoref{ex:rationallyconnected} shows that $\chi(H,\Oc_H)=1$. Alternatively, a direct computation shows that this is true in general (in any characteristic, for possibly singular $H$). This can be used to produce other sufficient conditions on a complete variety $X$ for the equality $d_X=1$, namely:

--- $X$ becomes isomorphic to such an $H$ after extension of the base field,

--- or $X$ is smooth and becomes birational to such a smooth $H$ after extension of the base field.
\end{example}

\begin{example}[Hypersurfaces]
\label{ex:hypersurfaces}
Let $H$ be a regular hypersurface of degree $p$ and dimension $\geq p-1$. By \autoref{ex:completeinter}, we have $\dep{H}=0$. Assume that $H$ has no closed point of degree prime to $p$. Then by \autoref{cor:cdim}, we have $\cdim_p(H) \geq p-1$. 

In case $\dim H=p-1$, this bound is optimal, and $H$ is $p$-incompressible (see also \cite[\S~7.3]{Mer-St-03} and \cite[Example~6.4]{Zai-09}). A more general statement was proved in \cite[Proposition~10.1]{firstst}. 

When $p=2$, we have $\cdim_2(H)\geq 1$. This bound is sharp when $\dim H=2$, as can be seen by taking for $H$ an anisotropic smooth projective Pfister quadric surface. In general this is far from being sharp: in characteristic not two, it is known that $\cdim_2(H)=2^n-1$, where $n$ is such that $2^n-1 \leq \dim H < 2^{n+1}-1$ (see \cite[\S~7]{Mer-St-03}).
\end{example}

\begin{example}[Separation]
Let $X$ and $Y$ be complete varieties, with $X$ regular. Assume that the degree of any closed point of $Y$ is divisible by $p$, that $\dep{X}=0$ (see in particular \autoref{ex:completeinter} and \autoref{ex:rationallyconnected}), and that
\[
\dim Y < p-1 \leq \dim X.
\]
Then the degree of any closed point of $Y_{k(X)}$ is divisible by $p$. (Indeed, assuming the contrary, there would be a correspondence $X \cor Y$ of degree prime to $p$. Then $\dep{Y}=0$ by \autoref{prop:cor}, hence $\nep{Y}=0$ by \autoref{th:main}.)
\end{example}

\begin{proposition}
\label{prop:fiber}
Let $f \colon Y \dasharrow X$ be a rational map between complete regular varieties. Let $F$ be the generic fiber of $f$, considered as a $k(X)$-variety. Then
\begin{enumerate}[label=(\roman*)]
\item \label{it:dep} If $\dep{F}=0$ then $\dep{X}=\dep{Y}$.

\item \label{it:nep} We have $\nep{X} \leq \nep{Y} \leq \nep{X} + \nep{F}$.
\end{enumerate}
\end{proposition}
\begin{proof}
Let us prove \ref{it:dep}. The map $f$ induces a correspondence $Y \xleftarrow{g} Z \xrightarrow{h} X$ with $g$ birational. By \autoref{lemm:main} we have $\de{Y}=\de{Z}$. 

We now prove the equality $\dep{X}=\dep{Z}$. We have a cartesian square
\[ \xymatrix{
F\ar[r]^\rho \ar[d]_l & Z \ar[d]^h \\ 
k(X) \ar[r]_\eta & X
}\]
Since $\dep{F}=0$, there is $v\in \G(F)\otimes \Zz_{(p)}$ such that 
\[
l_*v=1\in \Zz_{(p)} = \G(k(X))\otimes \Zz_{(p)}.
\]
Since $\rho^*\colon \G(Z) \to \G(F)$ is surjective (by the localisation sequence), we can find $u\in \G(Z)\otimes \Zz_{(p)}$ such that $\rho^*u=v$. Then we have
\[
\eta^* \circ h_*u=l_* \circ \rho^*u=l_*v=1=\eta^*[\Oc_X].
\]
It follows that the element $h_*u - [\Oc_X]$ is in the kernel of $\eta^*$, hence is nilpotent; therefore $h_*u$ is invertible in $\G(X)\otimes \Zz_{(p)}$. We can now conclude that $\dep{X}=\dep{Z}$, as in \autoref{lemm:main}.

The first inequality in \ref{it:nep} follows from  \autoref{prop:multiplicitychow}. To prove the second, note that $F \subset Y_{k(X)}$ and use \autoref{lemm:ffield} below.
\end{proof}

\begin{example}
Let $p$ be an odd prime. Let $H$ be a regular hypersurface of degree $p$ and dimension $p-1$, with $\nep{H}=1$. Let $X$ be a regular complete variety admitting a rational map $X \dasharrow H$, whose generic fiber $F$ is also a hypersurface of degree $p$ and dimension $p-1$. 

We have $\dep{F}=\dep{H}=0$ by \autoref{ex:completeinter}, and therefore $\dep{X}=0$ by \autoref{prop:fiber} \ref{it:dep}. Moreover $\nep{F} \leq 1$ by \autoref{th:main}, hence by \autoref{prop:fiber} \ref{it:nep}, we have $1 \leq \nep{X} \leq 2$. 

We use \autoref{cor:cdim}. If $\nep{X}=2$, then $\cdim_p(X)=\dim X=2(p-1)$. If $\nep{X}=1$, then $\cdim_p(X) \geq p-1$. This bound is sharp: if $\nep{F}=0$, then $X$ and $H$ are $p$-equivalent, hence have the same canonical $p$-dimension; but $\cdim_p(H)=\dim H=p-1$.
\end{example}

\section{Invariants of function fields}
\label{sect:ffields}

\begin{definition}
\label{def:function}
Let $K/k$ be a finitely generated field extension. Let $M$ be a regular projective variety such that $k(M)$ contains $K$ as a $k$-subalgebra with $[k(M):K]$ finite and prime to $p$. We define $\dep{K/k}:=\dep{M}$ and $\nep{K/k}:=\nep{M}$. By \autoref{lemm:indep} below, and \autoref{cor:main}, these integers do not depend on the choice of $M$. If there is no such $M$, we set $\dep{K/k}=\nep{K/k}=\infty$ (by Gabber's theorem, this may only happen in characteristic $p$).
\end{definition}

An alternative definition of $\nep{K/k}$ can be found in \cite[Remark~7.7]{Mer-St-03}. Note that, when $X$ is a complete variety, we have (by \autoref{prop:cor} and \autoref{prop:multiplicitychow})
\[
\dep{X} \leq \dep{k(X)/k} \quad \quad \text{ and } \quad \quad \nep{X} \leq \nep{k(X)/k},
\]
with equalities when $X$ is regular (by \autoref{lemm:indep}).

\begin{remark}
This process can be used more generally to define invariants $\M{H}{K/k}\otimes \Zz_{(p)}$, when the pair $(R, H)$ satisfies \autoref{def:main}.
\end{remark}

\begin{remark}[Geometrically unirational field extensions]
\label{prop:zero}
When $k$ has characteristic zero, it is possible to have some control on $\dep{K/k}$ without explicitly introducing a variety $M$ as in \autoref{def:function}. Namely assume that there is a field extension $l/k$ such that the ring $K \otimes_k l$ is contained in a purely transcendental field extension of $l$. Then $\dep{K/k}=0$ for any $p$. (Indeed by Hironaka's resolution of singularities \cite{Hir-64}, there is a smooth projective variety $M$ with function field $K$. Then $M$ is geometrically unirational, hence geometrically rationally connected, and $\de{M}=1$ by \autoref{ex:rationallyconnected}.)
\end{remark}

\begin{lemma}
\label{lemm:indep}
Let $X_1 \to X \leftarrow X_2$ be a diagram of varieties and proper morphisms, generically finite of degree prime to $p$. Then $X_1$ is $p$-equivalent to $X_2$.
\end{lemma}
\begin{proof}
Consider the cartesian square
\[ \xymatrix{
X'\ar[r]^{g_1} \ar[d]_{g_2} & X_2 \ar[d]^{f_2} \\ 
X_1 \ar[r]_{f_1} & X
}\]
We claim that $X'$ has an irreducible component which is generically finite of degree prime to $p$ over $X_1$; this component then gives a correspondence $X_1 \cor X_2$ of multiplicity prime to $p$, and the lemma follows by symmetry.

Since $f_i$ is generically finite and dominant, and $X_i$ is integral, the generic fiber of $f_i$ is the spectrum of $K_i:=k(X_i)$ (for $i=1,2$). The ring $B=K_1 \otimes_{k(X)} K_2$ is artinian, and its spectrum is the generic fiber of $g_2$. Write $B=B_1 \times \cdots \times B_n$, with $B_j$ an artinian local ring with residue field $L_j$ (for $j=1,\cdots,n$). We have
\[
[K_2:k(X)] = \dim_{K_1} B = \sum_{j=1}^n \dim_{K_1} B_j = \sum_{j=1}^n [L_j:K_1] \cdot \lgth B_j
\]
(for the last equality, use e.g.\ \cite[Lemma~A.1.3]{Ful-In-98}, replacing $A,B,M$ with $K_1,B_j,B_j$). But the leftmost integer is prime to $p$ by hypothesis. It follows that for some $j$ the integer $[L_j:K_1]$ is prime to $p$. The closure of $\Spec L_j$ in $X'$ then gives the required irreducible component.
\end{proof}

\begin{lemma}
\label{lemm:ffield}
Let $X$ and $M$ be complete varieties, with $M$ regular. Then
\[
\nep{X} \leq \nep{X_{k(M)}} + \nep{M}.
\]
\end{lemma}
\begin{proof}
Write $E=k(M)$. Let $L/E$ be a finite field extension such that $X_E(L) \neq \emptyset$ and $v_p[L\colon E]=\nep{X_E}$. Then the closure of the image of the map induced by an $L$-point of $X_E$
\[
\Spec L \hookrightarrow X_E=(\Spec E) \times_k X \to M \times_k X
\]
gives a correspondence $M \cor X$, whose multiplicity is $[L\colon E]$. By \autoref{prop:multiplicitychow}, we have
\[
\nep{X}\leq v_p[L\colon E] + \nep{M} = \nep{X_E} + \nep{M}.\qedhere
\]
\end{proof}

\begin{proposition}
\label{prop:rigidity}
Let $K/k$ be a finitely generated field extension. Let $X$ be a complete variety. Then
\[
\nep{X} \leq \nep{X_K} + \nep{K/k}.
\]
\end{proposition}
\begin{proof}
We may assume that there is a projective regular variety $M$ with $[k(M):K]$ finite and prime to $p$ (otherwise the statement is empty). Since $\nep{X_{k(M)}} \leq \nep{X_K}$, the statement follows from \autoref{lemm:ffield}.
\end{proof}

\begin{proposition}
\label{th:model}
Let $K/k$ be a finitely generated field extension. Then
\[
\nep{K/k} \leq \dep{K/k} + \Big[ \frac{\trdeg(K/k)}{p-1} \Big].
\]
\end{proposition}
\begin{proof}
As above, we may assume that there is a smooth projective variety $M$ such that $[k(M):K]$ is prime to $p$, and we apply \autoref{th:main} \ref{it:regular} with $X=M$.
\end{proof}

\begin{corollary}
\label{cor:model}
Let $K/k$ be a field extension of transcendence degree $<p-1$, with $\dep{K/k}=0$. Then $\nep{K/k}=0$.
\end{corollary}

\begin{corollary}
\label{cor:symbol}
Let $K/k$ be a field extension of transcendence degree $<p-1$. Assume that $\dep{K/k}=0$ (see in particular \autoref{prop:zero}). Then
\begin{enumerate}[label=(\roman*)]
\item \label{brauer} The relative Brauer group $\ker(\Br(k) \to \Br(K))$ has no $p$-primary torsion.

\item \label{symbol} Let $\alpha$ be a pure symbol in $K^M_*(k)/p$. Assume that $k$ has characteristic zero. If $\alpha_K=0$ then $\alpha=0$.
\end{enumerate}
\end{corollary}
\begin{proof}
We have $\nep{K/k}=0$ by \autoref{cor:model}. To prove \ref{brauer}, let $A$ be central simple $k$-algebra of $p$-primary exponent. Take for $X$ the Severi-Brauer variety of $A$. Then for any field extension $l/k$, the class of $A\otimes_k l$ vanishes in the Brauer group of $l$ if and only if $\nep{X_l}=0$. By \autoref{prop:rigidity}, we have $\nep{X_K}=\nep{X}$, and \ref{brauer} follows.

To prove \ref{symbol}, one can take for $X$ a complete generic $p$-splitting variety \cite{Norm-varieties} for $\alpha$, and argue as above.
\end{proof}

\begin{example}
\label{ex:final}
Let $D$ be a non-trivial $p$-primary central division $k$-algebra, and $K/k$ a splitting field extension for $D$. Assume that $K$ is the function field of a complete, smooth, geometrically rational variety (or more generally that $\dep{K/k}=0$). \autoref{cor:symbol}, \ref{brauer} says that $\trdeg(K/k) \geq p-1$. One may ask whether we always have 
\[
\trdeg(K/k) \geq \ind(D)-1.
\]
In other words, has the Severi-Brauer variety of $D$ the smallest possible dimension among the complete, smooth, geometrically rational varieties whose function field splits $D$?
\end{example}

{\bf Acknowledgements.} I would like to thank Nikita Karpenko for his comments on this paper, and the referee for his/her careful reading of the paper and constructive remarks. The first version of this paper has been written while I was employed by the University of Nottingham.

\bibliographystyle{alpha}

\end{document}